\documentclass{article}
\usepackage{amsmath,amssymb,amsthm, amscd}
\usepackage{pifont,mathrsfs}
\usepackage{array,lastpage}
\usepackage{enumerate,xspace,pifont,shadow}
\usepackage{mathrsfs}
\usepackage[T1]{fontenc}

\DeclareSymbolFont{AMSb}{U}{msb}{m}{n}
\DeclareMathSymbol{\Z}{\mathbin}{AMSb}{"5A}
\DeclareMathSymbol{\R}{\mathbin}{AMSb}{"52}
\DeclareMathSymbol{\N}{\mathbin}{AMSb}{"4E}
\DeclareMathSymbol{\Q}{\mathbin}{AMSb}{"51}

\newcommand{\C}{\mathfrak{C}}
\newcommand{\tp}{\textup{tp}}
\newcommand{\stp}{\textup{stp}}

\newcommand{\acl}{\textup{acl}}
\newcommand{\Th}{\textup{Th}}
\newcommand{\dom}{\textup{dom}}

\newcommand{\Aut}{\textup{Aut}}

\newcommand{\cl}{\textup{cl}}
\newcommand{\coker}{\textup{coker}}

\newcommand{\id}{\textup{id}}

\newcommand{\G}{\overline{G}}

\def\Ind{\setbox0=\hbox{$x$}\kern\wd0\hbox to 0pt{\hss$\mid$\hss}
\lower.9\ht0\hbox to 0pt{\hss$\smile$\hss}\kern\wd0}

\def\Notind{\setbox0=\hbox{$x$}\kern\wd0\hbox to 0pt{\mathchardef
\nn=12854\hss$\nn$\kern1.4\wd0\hss}\hbox to
0pt{\hss$\mid$\hss}\lower.9\ht0 \hbox to
0pt{\hss$\smile$\hss}\kern\wd0}

\def\ind{\mathop{\mathpalette\Ind{}}}

\def\nind{\mathop{\mathpalette\Notind{}}}

\newtheorem{thm}{Theorem}[section]
\newtheorem{lem}[thm]{Lemma}
\newtheorem{cor}[thm]{Corollary}
\newtheorem{prop}[thm]{Proposition}

\newtheorem{fact}[thm]{Fact}

\newtheorem{quest}[thm]{Question}

\theoremstyle{definition}
\newtheorem{definition}[thm]{Definition}

\theoremstyle{remark}
\newtheorem{remark}[thm]{Remark}

\theoremstyle{remark}
\newtheorem{example}[thm]{Example}

\theoremstyle{remark}
\newtheorem{claim}[thm]{Claim}

\theoremstyle{remark}

\theoremstyle{remark}

\begin{document}
\bibliographystyle{amsplain}

\title{The Schr\"{o}der-Bernstein property\\for weakly minimal theories}
\author{John Goodrick and Michael C. Laskowski}

\maketitle

\begin{abstract}
For a countable, weakly minimal theory $T$, we show that the Schr\"{o}der-Bernstein property (any two elementarily bi-embeddable models are isomorphic) is equivalent to each of the following:

1. For any $U$-rank-$1$ type $q \in S(\acl^{eq}(\emptyset))$ and any automorphism $f$ of the monster model $\mathfrak{C}$, there is some $n < \omega$ such that $f^n(q)$ is not almost orthogonal to $q \otimes f(q) \otimes \ldots \otimes f^{n-1}(q)$;

2. $T$ has no infinite collection of models which are pairwise elementarily bi-embeddable but pairwise nonisomorphic.

We conclude that for countable, weakly minimal theories, the Schr\"{o}der-Bernstein property is absolute between transitve models of ZFC.
\end{abstract}

\section{Introduction}

We are concerned the following property of a first-order theory $T$:

\begin{definition}
A theory $T$ has the \emph{Schr\"{o}der-Bernstein property}, or the \emph{SB~property}, if any two elementarily bi-embeddable models of $T$ are isomorphic.
\end{definition}

Our motivation is to find some nice model-theoretic characterization of the class of complete theories with the SB~property.  This property was first studied in the 1980's by Nurmagambetov in \cite{nur2} and \cite{nur1} (mainly within the class of $\omega$-stable theories).  In \cite{nur1}, he showed:

\begin{thm}
\label{ttrans}
If $T$ is $\omega$-stable, then $T$ has the SB~property if and only if $T$ is nonmultidimensional.
\end{thm}

One of the results from the thesis of the first author (\cite{my_thesis}) was:

\begin{thm}
If a countable complete theory $T$ has the SB~property, then $T$ is superstable, nonmultidimensional, and NOTOP, and $T$ has no \emph{nomadic types}; that is, there is no type $p \in S(M)$ such that there is an automorphism $f \in \Aut(M)$ for which the types $\left\{f^n(p) : n \in \N \right\}$ are pairwise orthogonal.
\end{thm}

In particular, any countable theory with the SB~property must be classifiable (in the sense of Shelah).  Within classifiable theories, the SB~property seems to form a new dividing line distinct from the usual dichotomies in stability theory.

In this note, we investigate the SB~property for weakly~minimal theories (that is, theories in which the formula ``$x=x$'' is weakly minimal).  We prove the following characterization, confirming a special case of a conjecture of the first author (from \cite{my_thesis}):

\begin{thm}
\label{wm_SB}
If $T$ is countable and weakly minimal, then the following are equivalent:

1. $T$ has the SB~property.

2. For any $U$-rank-$1$ type $q \in S(\acl^{eq}(\emptyset))$ and any automorphism $f$ of the monster model $\mathfrak{C}$, there is some $n < \omega$ such that $f^n(q)$ is not almost orthogonal to $q \otimes f(q) \otimes \ldots \otimes f^{n-1}(q)$.

3. $T$ has no infinite collection of models which are pairwise elementarily bi-embeddable but pairwise nonisomorphic.
\end{thm}

The proof of Theorem~\ref{wm_SB} uses some geometric stability theory to reduce to the case where $p$ is the generic type of an infinite definable group, and then a Dushnik-Miller style argument can be used to construct witnesses to the failure of the SB~property whenever condition~2 fails.

A corollary is that for countable, weakly minimal theories, the SB~property is invariant under forcing extensions of the universe of set theory:

\begin{cor}
Among countable weakly~minimal theories, the SB~property is absolute between transitive models of ZFC containing all the ordinals.

\end{cor}

\begin{proof}

First, note that for a countable theory $T$, condition~(2) of Theorem~\ref{wm_SB} is equivalent to the following statement:

For any \emph{countable} model $M$ of $T$, any $1$-type $q \in S(M)$, and any $f \in \Aut(M)$, there is an $n < \omega$ such that $f^n(q)$ is not almost orthogonal to $q \otimes f(q) \otimes \ldots \otimes f^{n-1}(q)$.

Using this, it follows that among countable theories, condition~(2) is a (lightface) $\Pi^1_1$ property.  So the corollary follows from the Shoenfield Absoluteness Theorem for $\Pi^1_2$ relations (Theorem~98 of \cite{jech}).

\end{proof}

\begin{cor}
If $T$ is any countable weakly~minimal theory and $T' \supseteq T$ is the expansion by new constants with one new constant naming each element of $\acl_T^{eq}(\emptyset)$, then $T'$ has the SB~property.
\end{cor}

\begin{proof}

$T'$ trivially satisfies condition~2 of Theorem~\ref{wm_SB} (and is still countable and weakly~minimal).

\end{proof}

After discussing some preliminaries in section~2, we show in section~3 that the SB~property is incompatible with the existence of a definable group with a sufficiently ``generic'' automorphism (Theorem~\ref{wm_group_SB}).  In section~4, we give another criterion for the failure of SB (Theorem~\ref{nomadic_types}) and finish with a proof of Theorem~\ref{wm_SB}.

\section{Preliminaries}

We follow the usual conventions of stability theory, as explained in \cite{pillay}.  

First, by the canonicity of Shelah's $eq$-construction, it is straightforward to see:

\begin{fact}
If $T$ is any theory, then $T$ has the SB~property if and only if $T^{eq}$ does.
\end{fact}

\textbf{For the rest of this paper, we will assume that $T = T^{eq}$ for all the theories $T$ we consider.}

A technical advantage of working with weakly~minimal theories is that models are easy to construct:

\begin{lem}
\label{wm_models}
If $T$ is weakly~minimal, $M \models T$, and $A \subseteq \mathfrak{C}$, then $$\acl(M \cup A) \models T.$$
\end{lem}

\begin{proof}
By the Tarski-Vaught test, it is enough to check that any consistent formula $\varphi(x; \overline{b})$ over $M \cup A$ in a \emph{single} free variable has a realization in $\acl(M \cup A)$.  If $\varphi(x; \overline{b})$ has a realization $a \in \acl(\overline{b})$, then we are done.  Otherwise, by weak minimality of $T$, any realization $a$ of the formula is independent from $\overline{b}$.  Let $p = \stp(\overline{b})$.  Then $$\mathfrak{C} \models \forall x \left[ \varphi(x; \overline{b}) \leftrightarrow d_p \overline{y} \varphi(x; \overline{y}) \right],$$ and the formula $d_p \overline{y} \varphi(x; \overline{y})$ is definable over $\acl(\emptyset)$, so it is realized in $M$.
\end{proof}

The next lemma is true in \emph{any} theory (not only weakly minimal ones).

\begin{lem}
\label{alg_parameters}
Suppose that $T$ has an infinite collection of models which are pairwise nonisomorphic and pairwise elementarily bi-embeddable, and $a \in \acl(\emptyset)$.  Then the expansion $T_a := \Th(\mathfrak{C}, a)$ with a new constant naming $a$ does not have the SB~property.  In fact, $T_a$ also has an infinite collection of pairwise nonisomorphic, pairwise bi-embeddable models.
\end{lem}

\begin{proof}
Let $\{M_i : i < \omega\}$ be an infinite collection of models of $T$, pairwise nonisomorphic and pairwise bi-embeddable, and let $n$ be the number of distinct realizations of $\tp(a)$.  Then we claim that among any $n + 1$ of the $M_i$'s -- say, $M_0, \ldots, M_{n}$ -- there are two that are bi-embeddable as models of $T_a$.  To see this, first pick elementary embeddings $f_k : M_0 \rightarrow M_k$ and $g_k : M_k \rightarrow M_0$ for every $k$ with $1 \leq k \leq n$.  Without loss of generality, every map $f_k \circ g_k$ fixes $a$, since we can replace $f_k$ by $(f_k \circ g_k)^i \circ f_k$ for some $i$ such that $(f_k \circ g_k)^{i+1} (a) = a$.  By the pigeonhole principle, there must be two distinct $k, \ell \leq n$ such that $g_k$ and $g_\ell$ map $a$ onto the same element $a'$.  Thus $f_k (a') = f_\ell (a') = a$, and so the maps $$f_k \circ g_\ell : M_\ell \rightarrow M_k$$ and $$f_\ell \circ g_k: M_k \rightarrow M_\ell$$ both fix $a$.  Since $M_k \ncong M_\ell$, they are not isomorphic as models of $T_a$ either.
\end{proof}

\section{Weakly minimal groups}

In this section, we consider the relation between weakly~minimal groups in a countable language and the SB~property.  We show that if $T$ is any weakly~minimal theory in which there is a definable weakly~minimal abelian group with a certain kind of ``generic'' automorphism, then $T$ does not have the SB~property.  One of the key lemmas is a variation of the Baire~category theorem (Lemma~\ref{generic_elements}).

\begin{fact}
\label{abelian_by_finite}
(\cite{pillay}) If $(G; \cdot, \ldots)$ is an $\emptyset$-definable weakly~minimal group, then $G$ has an $\emptyset$-definable abelian subgroup $H$ of finite index.
\end{fact}

Throughout this section, we assume that $(G; +)$ is a weakly~minimal \emph{abelian} group which is $\emptyset$-definable in the \emph{countable} theory $T$.  The group $G$ is equipped with all the definable structure induced from $T$, which may include additional structure that is not definable from the group operation alone.  Fact~\ref{abelian_by_finite} shows the assumption that $G$ is abelian is not too strong.

We make the following additional assumptions:

1. $G$ is saturated.  We identify the set of all strong types over $0$ with the set of points in $\G := G / G^{\circ}$, and we refer to subsets $X$ of $\G$ as being dense, open, etc. if the corresponding subsets of the Stone space are.

2. The connected component $G^\circ$ is the intersection of the $\emptyset$-definable groups $G_0 > G_1 > G_2 > \ldots$, each of which is a subgroup of $G$ of finite index.

3. $G_{i+1} \neq G_i$.  (This assumption will be justified later.)

We let $\Aut(G)$ denote the group of all \emph{elementary} bijections from $G$ to $G$ (not just the group of all group automorphisms), and we let $\Aut(\G)$ be the group of all group automorphisms of $\G$ which are induced by maps in $\Aut(G)$.

We recall some important facts about the definable structure of a weakly minimal group $G$.  The forking relation between generic elements is controlled by the action of the division ring $D$ of definable quasi-endomorphisms of $G$ (see \cite{pillay}).  Any nonzero quasi-endomorphism $d \in D$ is represented by a definable subgroup $S_d \leq G \times G$ which is an ``almost-homomorphism,'' that is, the projection of $S_d$ onto the first coordinate is a subgroup of finite index and the cokernel $\{g \in G : (0,g) \in S_d\}$ is finite.  In our context, we may take $S_d$ to be $\acl(0)$-definable or even $0$-definable, so $D$ is countable.  From this point on, we fix some such $0$-definable $S_d$ representing each $d \in D$; the particular choice of $S_d$ will not turn out to matter for our purposes.

If $d \in D \setminus \{0\}$ and $H$ is any subgroup of $\acl(0)$ containing $\ker(S_d) \cup \coker(S_d)$, then $S_d$ naturally induces an injective map $\overline{S_d}$ from a subgroup of $K \leq G/H$ of finite index into $G/H$; by extension, we think of this as $d$ itself acting on $K$, and write this map as ``$\overline{d}_H$'' or simply ``$\overline{d}$.''

Note that $\G$ is a Polish group under the Stone topology.  (The fact that the cosets of the groups $G_i / G^\circ$ form a base  for the topology ensures that the space is separable.)  Condition~3 implies that $\G$ is perfect (i.e. there are no isolated points).  We will repeatedly use the fact that any $\varphi \in \Aut(\G)$ (or indeed any image of an elementary embedding from $G$ into itself) is continuous with respect to this topology.

\begin{example}
Let $G$ be the direct product of $\omega$ copies of the cyclic group $\Z_p$, with its definable structure given by the group operation $+$ and unary 
predicates for each of the subgroups $H_i$ consisting of all elements of $g$ whose $i$th coordinate is zero.  Then if we let $G_i$ be the intersection of the groups $H_0, \ldots, H_{i-1}$, we are in the situation above, with $D \cong \mathbb{F}_p$.  Note that although $| \Aut(\G) | = 2^{\aleph_0} $, every $\overline{\varphi} \in \Aut(\G)$ has the property that $\varphi^{p-1} = \id$, so in the terminology of Definition~\ref{unipotent} below, every automorphism is unipotent.
\end{example}

Now suppose that $H$ is a \emph{finite} $0$-definable subgroup of $G$.  Then if $\G_H$ denotes $G / (G^\circ + H)$, we can quotient by the projection map $\pi: \G \rightarrow \G_H$ to define a topology on $\G_H$.  A crucial observation for what follows is that $\G_H$ is still a perfect Polish group.  Note that the finiteness of $H$ lets us conclude that $\G_H$ is perfect; if, say, $(H + G^\circ) / G^\circ$ were dense in $\G$, then the topology on $\G_H$ would be trivial.

Some more notation: fix some finite $0$-definable $H \leq G$.  If $\varphi \in \Aut(G)$, then $\varphi^*_H : \G_H \rightarrow \G_H$ is the corresponding automorphism of $\G_H$.  We may write this as simply ``$\varphi^*$'' if $H$ is understood.  If $g$ is an element of $G$ or $\G$, then $\overline{g}_H$ is its image in $G / (G^\circ + H)$ under the natural quotient projection.  If $d \in D$ and there is some $0$-definable $S_d$ representing $d$ such that $\coker(S_d) \subseteq H$, then $d^*_H$ (or $d^*$) is the corresponding partial function on $\G_H$.  Note that in computing $d^*_H$, the particular choice of $S_d$ representing $d$ only affects the domain of $d^*_H$; two different choices of $S_d$ result in partial maps on $\G_H$ which agree on their common domain, and this common domain is a subgroup of $\G_H$ of finite index.  This motivates the following:

\begin{definition}
If $f_1$ and $g_2$ are two group homomorphisms from open subgroups $K_1, K_2 \leq \G_H$ into $\G_H$, then we write ``$f_1 =^* f_2$'' if there is some open subgroup $K'$ of $K_1 \cap K_2$ such that $f_1 \upharpoonright K' = f_2 \upharpoonright K'$.
\end{definition}

\begin{definition}
\begin{enumerate}
\item If $D_0 \subseteq D$ is finite, then $H$ is \emph{good for $D_0$} if $H$ is a finite $0$-definable subgroup of $G$ containing $\coker(S_d)$ for every $d \in D_0$.
\item If $q \in D[x]$, then $H$ is \emph{good for $q$} if $H$ is good for the set of coefficients of $q$.
\end{enumerate}
\end{definition}

\begin{definition}
If $X \subseteq \G$, then we say that $g + G^\circ$ is in the \emph{$D$-closure of $X$}, or $\cl_D(X)$, if there are:

1. Elements $h_1 + G^\circ, \ldots, h_n + G^\circ$ of $X$,

2. Elements $d_1, \ldots, d_n$ of $D$, and

3. A subgroup $H \leq G$ which is good for $\{d_1, \ldots, d_n\}$,

such that $g_H = (d_1)^*_H (\overline{h_1}_H) + \ldots + (d_n)^*_H (\overline{h_n}_H)$.  

The set $X \subseteq \G$ is \emph{$D$-closed} if $X = \cl_D(X)$.
\end{definition}

If $q = \Sigma_{i \leq n} d_i x^i$ is a polynomial in $x$ over $D$, and $H \leq G$ is good for $q$, then $$q^*_H \varphi := \Sigma_{i \leq n} (d_i)^*_H \circ (\varphi^*_H)^i.$$  Note that $q^*_H$ is a continuous group map from a finite-index subgroup of $\G_H$ into $\G_H$, and that $(q + r)^*_H =^* q^*_H + r^*_H$ and $(q \cdot r)^*_H =^* q^*_H \circ r^*_H$.

\begin{definition}
\label{unipotent}
\begin{enumerate}
\item $\overline{\varphi} \in \Aut(\G)$ is \emph{unipotent} if there is some nonzero $n \in \omega$ such that ${\overline{\varphi}}^n = \id$.
\item $\overline{\varphi} \in \Aut(\G)$ is \emph{weakly generic} if for every $q(x) \in D[x] \setminus \left\{0\right\}$ and every $H \leq G$ which is good for $q$, the map $q^*_H \varphi$ is not identically zero on its domain.
\item $\overline{\varphi} \in \Aut(\G)$ is \emph{everywhere generic} if for every $q(x) \in D[x] \setminus \left\{0\right\}$, for every $H \leq G$ which is good for $q$, and for every nonempty open $U \subseteq \dom(q^*_H \varphi)$, the map $q^*_H \varphi \upharpoonright U$ is not identically zero.

\end{enumerate}

\end{definition}

\begin{prop}
\label{everywhere_generic}
If $\overline{\varphi} \in \Aut(\G)$ is weakly generic, then $\overline{\varphi}$ is everywhere generic.
\end{prop}

\begin{proof}
Suppose that $\overline{\varphi}$ is \emph{not} everywhere generic, as witnessed by $q \in D[x] \setminus \left\{0\right\}$, $H$, and a nonempty open $U \subseteq \dom(q^*_H \varphi)$ such that $q^*_H \varphi \upharpoonright U$ is identically zero.  Say $q = x^m \cdot q_0$, where $q_0$ has a nonzero constant term; then since $(\varphi^*_H)^m$ is an injective group homomorphism, $(q_0)^*_H \varphi \upharpoonright (\varphi^*_H)^m(U)$ is identically zero and $(\varphi^*_H)^m(U)$ is open; so we may assume that $x$ does not divide $q$.  Write $q = d_k x^k + \ldots + d_0$, where $d_i \in D$.  Since $d_0 \neq 0$, it follows that if $r = d_0^{-1} d_k x^k - \ldots - d_0^{-1} d_1 x$, then there is some nonempty open $V \subseteq \G_H$ such that $(d_0^{-1})^* (V) \subseteq U$ and $r^*_H \varphi \upharpoonright V = \id_V$.

Without loss of generality, $V = (g + G_n + H) / (G^\circ + H)$ for some $g \in G$ and some $n < \omega$, and (shrinking $V$ if necessary) we may also assume that $(G_n + H) / (G^\circ + H) \subseteq \dom(r^*_H \varphi)$.  For any $h \in (G_n + H) / (G^\circ + H)$, there are $g_1, g_2 \in V$ such that $h = g_1 + g_2$, so it follows that $r^*_H \varphi \upharpoonright (G_n + H) / (G^\circ + H)$ is the identity map.  For any $\ell \in \omega$, the map $(r^\ell)^*_H \varphi$ induces a group homomorphism from a subgroup of $G / (G_n + H)$ into $G / (G_n + H)$, and since $G / (G_n + H)$ is finite, there are numbers $\ell < m < \omega$ such that the image of $(r^{\ell} - r^m)^*_H \varphi$ is contained in $(G_n + H) / (G^\circ + H)$.  Let $s = (r^{\ell} - r^m)^2$.  Note that since $x$ divides $r$, the polynomial $s$ is nonzero, but $s^*_H \varphi$ is identically zero on its domain; so $\varphi$ is not weakly generic.

\end{proof}

\begin{quest}
If $\Aut(\G)$ contains a non-unipotent element, does $\Aut(\G)$ necessarily contain a weakly generic element?
\end{quest}

\begin{definition}
Let $S$ be a Polish space.  A continuous function $f: S^k \rightarrow S$ is \emph{nondegenerate} if there is some $i$ ($1 \leq i \leq k$) such that for any elements $a_i, \ldots, a_{i-1}, a_{i+1}, \ldots, a_k \in S$, the function $$f_{\overline{a}}(x) = f(a_1, \ldots, a_{i-1}, x, a_{i+1}, \ldots, a_k)$$ is a homeomorphism.
\end{definition}

The next lemma is a version of the Baire category theorem.  (To get the usual Baire category theorem, let $f_0 : S \rightarrow S$ be the identity map.)

\begin{lem}
\label{generic_elements}
Suppose that $S$ is a perfect Polish space and $\langle f_i : i \in \omega \rangle$ is a countable collection of continuous, nondegenerate functions, with $f_i : S^{k(i)} \rightarrow S$, and $\langle H_\ell : \ell \in \omega \rangle$ is a collection of nowhere dense subsets of $S$.  Then there is a nonempty perfect set of elements $\langle a_\sigma : \sigma \in 2^{\aleph_0} \rangle$ of $S$ such that for every $i, \ell \in \omega$ and every set $\left\{ \sigma_1, \ldots, \sigma_{k(i)} \right\}$ of $k(i)$ \emph{distinct} elements of $2^{\aleph_0}$, $f_i(a_{\sigma_1}, \ldots, a_{\sigma_{k(i)}}) \notin H_\ell$.
\end{lem}

\begin{proof}

First, we set some notation.  Let $\mathcal{O}$ be the set of all \emph{nonempty} open subsets of $S$.  A \emph{condition} is a function $F: D \rightarrow \mathcal{O}$, where $D$ is some finite, downward-closed subset of $2^{< \omega}$, such that

1. If $s^{\frown} \langle 0 \rangle, s^{\frown} \langle 1 \rangle \in D$, then $F(s^{\frown} \langle 0 \rangle) \cap F(s^{\frown} \langle 1 \rangle) = \emptyset$; and

2. If $s$ is an initial segment of $t$, then $F(t) \subseteq F(s)$.

Given two conditions $F, F'$, we write $F \leq F'$ if $\dom(F) \subseteq \dom(F')$ and for any $s \in \dom(F)$, $F'(s) \subseteq F(s)$.

A \emph{viable triple} is an ordered triple $\langle i, \ell, (s_1, \ldots, s_k) \rangle$ such that $i, \ell \in \omega$, the $s_j$'s are pairwise incompatible elements of $2^{< \omega}$, and $k = k(i)$.  Let $$\left\{ \langle i(t), \ell(t), (s^t_1, \ldots, s^t_{k(t)}) \rangle : t \in \omega \right\}$$ be an enumeration of all viable triples.

We construct an increasing sequence of conditions $\langle F(t) : t < \omega \rangle$ by induction on $t$.  As a base case, let $F(0)$ be the function with domain $\left\{\langle \rangle \right\}$ such that $F(\langle \rangle) = S$.  For the induction step, suppose that we have picked $F(t)$.

\begin{claim}
We can pick $F(t+1) \geq F(t)$ such that $$\{s^t_1, \ldots, s^t_{k(t)} \} \subseteq \dom(F(t+1))$$ and $$f_{i(t)}(F(t+1)(s^t_1), \ldots, F(t+1)(s^t_{k(t)})) \cap H_{\ell(t)} = \emptyset. $$
\end{claim}

\begin{proof}
First, since $S$ is perfect, we can pick a condition $F' \geq F(t)$ such that $\{s^t_1, \ldots, s^t_{k(t)} \} \subseteq \dom(F')$.  By the fact that $f_{i(t)}$ is nondegenerate and $H_{\ell(t)}$ is nowhere dense, there is a tuple $(a_1, \ldots, a_{k(t)})$ such that $a_m \in F'(s^t_m)$ and $f_{i(t)} (a_1, \ldots, a_{k(t)})\notin \overline{H_{\ell(t)}}$.  Pick an open neighborhood $U$ of $f_{i(t)}(a_1, \ldots, a_{k(t)})$ such that $U \cap \overline{H_{\ell(t)}} = \emptyset$.  By continuity, $f^{-1}_{i(t)}(U)$ is open, and it contains $(a_1, \ldots, a_{k(t)})$.  Therefore, there is a condition $F(t+1) \geq F'$ such that $a_m \in F(t+1)(s^t_m)$ and $$f_{i(t)}(F(t+1)(s^t_1), \ldots, F(t+1)(s^t_{k(t)})) \subseteq U,$$ so this $F(t+1)$ works.

\end{proof}

Pick $F(t+1)$ by induction as in the Claim.  Note that it follows that $$\bigcup_{t \in \omega} \dom(F(t)) = 2^{< \omega},$$ since every $s \in 2^{< \omega}$ is in a viable triple.  For any $\sigma \in 2^\omega$, let $$\widehat{F}(\sigma) = \bigcap \{F(t)(\sigma \upharpoonright m) : t \in \omega \textup{ and } \sigma \upharpoonright m \in \dom(F(t)) \}.$$  It follows from the properties of the conditions $F(t)$ that $\widehat{F}(\sigma)$ is always well-defined and nonempty and that if $\sigma \neq \tau$, then $\widehat{F}(\sigma) \cap \widehat{F}(\tau) = \emptyset$.  Finally, pick elements $a_\sigma \in \widehat{F}(\sigma)$ for every $\sigma \in 2^\omega$.  The fact that these $a_\sigma$'s work follows from the way we enumerated the viable triples and the Claim above.

\end{proof}

We need one more simple lemma before proving the main theorem of this section.

\begin{lem}
\label{loc_mod_formula}
Suppose that $T$ is \emph{any} stable theory, $M \models T$, $\theta(x)$ is a weakly minimal formula over $\emptyset$ in $T$, and $A \subseteq \theta(\mathfrak{C})$.  Then if $M' = \acl(M \cup A)$ and $M' \models T$, $$\theta(M') \subseteq \acl(\theta(M) \cup A).$$
\end{lem}

\begin{proof}
Suppose $b \in \theta(M')$ and $b \in \acl(M \cup \{a_1, \ldots, a_n\})$, where $a_i \in A$ and $n$ is minimal.  Minimality of $n$ implies that $\{a_1, \ldots, a_n\}$ is independent over $M$ and $b$ is interalgebraic with $a_n$ over $M \cup \{a_1, \ldots, a_{n-1}\}$.  Note that $\tp(b a_1 \ldots a_n / M)$ is finitely satisfiable in $\theta(M)$, and therefore $$b a_1 \ldots a_n \ind_{\theta(M)} M.$$  So $$U(b a_1 \ldots a_n / \theta(M)) = U(b a_1 \ldots a_n / M) = n,$$ and therefore $b$ is interalgebraic with $a_n$ over $\theta(M) \cup \{a_1, \ldots, a_{n-1}\}$.
\end{proof}

\begin{thm}
\label{wm_group_SB}
Suppose that $T$ is weakly minimal and countable, and that $G$ is a weakly~minimal abelian group which is $\emptyset$-definable in $T$.  If $\Aut(\G)$ contains a weakly generic map $\overline{\varphi}$, then $T$ has an infinite collection of pairwise nonisomorphic, pairwise elementarily bi-embeddable models.
\end{thm}

\begin{proof}

First we note that if $\Aut(\G)$ contains a weakly generic map, then $\G$ must be infinite, so we can pick the definable subgroups $G_0 > G_1 > \ldots$ so that $G_{i+1} \neq G_i$.  This justifies assumption~3 at the beginning of the section and implies that $\G$ is a perfect topological space.

The key to our proof is that we can use the fact that $|\Aut(\G)| = 2^{\aleph_0}$ to construct bi-embeddable models $M$ and $N$ of $T$ via chains of length $2^{\aleph_0}$ where at each successor stage we kill one potential isomorphism between $\G(M)$ and $\G(N)$.  We call this a Dushnik-Miller type argument since it recalls the idea of the proof of Theorem~5.32 in \cite{dushnik_miller}.

As usual, let $D$ be the division ring of definable quasi-endomorphisms of $G$.  Fix a weakly generic $\overline{\varphi} \in \Aut(\G)$, which is also everywhere generic by Proposition~\ref{everywhere_generic}.  Also pick some $f \in \Aut(\mathfrak{C})$ such that $\overline{f \upharpoonright G(\mathfrak{C})} = \overline{\varphi}$.  

Since $\G$ is separable, $|\Aut(\G)| \leq 2^{\aleph_0}$.  Therefore we can pick a sequence $\{ (h_\alpha, i(\alpha), j(\alpha)) : \alpha < 2^{\aleph_0} \}$ listing triples in $\Aut(\G) \times \omega \times \omega$ in such a way that for any $\alpha < 2^{\aleph_0}$,

(A) $i(\alpha) < j(\alpha)$, and

(B) For any $i_0 < j_0 < \omega$ and any $h \in \Aut(\G)$, there is some $\beta$ such that $\alpha < \beta < 2^{\aleph_0}$ and $(h_\beta, i(\beta), j(\beta)) = (h, i_0, j_0)$.

Next, we will define models $M^\ell_\alpha$ of $T$ and subsets $X^\ell_\alpha$ of $\G$ for every $\ell < \omega$ and $\alpha < 2^{\aleph_0}$ by recursion on $\alpha < 2^{\aleph_0}$ such that:

\begin{enumerate}
\item $| M^\ell_\alpha \cup X^\ell_\alpha | \leq |\alpha| + \aleph_0$;
\item $M^0_\alpha \succ M^1_\alpha \succ M^2_\alpha \succ \ldots$ and $f(M^\ell_\alpha) \prec M^{\ell+1}_\alpha$;
\item If $\alpha < \beta$, then $M^\ell_\alpha \prec M^\ell_\beta$ and $X^\ell_\alpha \subseteq X^\ell_\beta$;
\item $\G(M^\ell_\alpha) \cap X^\ell_\alpha = \emptyset$;
\item If $\alpha < \beta$, then at least one of the following holds:

(C) There is some $a \in \G(M^{i(\alpha)}_\beta)$ such that $h_\alpha(a) \in X^{j(\alpha)}_\beta$; or 

(D) There is some $b \in \G(M^{j(\alpha)}_\beta)$ such that $h^{-1}_\alpha(b) \in X^{i(\alpha)}_\beta$.
\end{enumerate}

Once we have the models $M^\ell_\alpha$, we can let $M^\ell = \bigcup_{\alpha < 2^{\aleph_0}} M^\ell_\alpha$.  The $M^\ell$'s are pairwise bi-embeddable (via inclusions and iterates of $f$), and properties~$4$ and $5$ ensure that if $\ell \neq k$ then there is no $h \in \Aut(\C)$ mapping $\overline{G}(M^\ell)$ onto $\overline{G}(M^k)$, so \textit{a fortiori} there is no such $h$ mapping $M^\ell$ onto $M^k$.  Thus $\{M^\ell : \ell < \omega\}$ will be the models we are looking for.

For the base case $\alpha = 0$, we can pick some countable $M \models T$ such that $f(M) \subseteq M$ and let $M^\ell_0 = M$ and $X^\ell_0 = \emptyset$.  At limit stages we take unions, and there are no problems.  So we only have to deal with the successor stage, and suppose we have $M^\ell_\alpha$ and $X^\ell_\alpha$ as above.

To set some more notation, if $Z \subset \G$ and $H \leq G$, then we let $$Z_H = \left\{h \in \G_H : \exists g \in G \,\left[g + G^\circ \in Z \textup{ and } h = g + (G^\circ + H) \right]\,\right\}.$$  Here are two possible situations which we will consider:

$(*)$ There is a nonempty clopen \emph{subgroup} $K \leq \G$ with the following property: for every $q(x) \in D[x]$ such that $x | q$, every nonempty open $U \subseteq K$, and every finite $0$-definable $H \leq G$ which is good for $q$, there is a $g \in U_H$ such that either $g \notin \dom (q^*_H \varphi)$ or $(h_\alpha)^*_H(g) \neq q^*_H \varphi (g)$.

$(\dagger)$ For every $q(x) \in D[x]$, every nonempty open $U \subseteq \G$, and every finite $0$-definable $H \leq G$ which is good for $q$, there is a $g \in U_H$ such that either $g \notin \dom(q^*_H \varphi)$ or $(h^{-1}_\alpha)^*_H (g) \neq q^*_H \varphi (g)$.

\begin{claim}
Either $(*)$ holds or $(\dagger)$ holds (or possibly both).
\end{claim}

\begin{proof}
The failure of $(\dagger)$ gives us a nonempty open $U \subseteq \G$, a $q(x) \in D[x]$, and $H \leq G$ which is good for $q$ such that for every $g \in U_H$, $(h^{-1}_\alpha)^*_H (g) = q^*_H \varphi (g)$.  Let $K$ be the subgroup of $\G$ generated by $U$.  The continuity of the group operation implies that $K$ is open, in fact clopen, and since $(h^{-1}_\alpha)^*_H$ and $q^*_H \varphi$ are both homomorphisms from $\G_H$ into itself, \begin{equation}\label{eq:h-inv}(h^{-1}_\alpha)^*_H \upharpoonright K_H = q^*_H \varphi \upharpoonright K_H.\end{equation}  Let $K' = h^{-1}_\alpha(K)$, so $K'$ is another clopen subgroup of $\G$.  By the failure of $(*)$, there is an $r(x) \in D[x]$ such that $x | r$, some nonempty open $V \subseteq K'$, and some $H' \leq G$ good for $r$ such that for every $g \in V_H$, \begin{equation}\label{eq:h}(h_\alpha)^*_{H'}(g) = r^*_{H'} \varphi (g).\end{equation}

Note that equations~\ref{eq:h-inv} and \ref{eq:h} remain true if we replace either $H$ or $H'$ with any larger finite $0$-definable subgroup of $G$, such as $\widehat{H} := H + H'$.  For any $g \in (h_\alpha)^*_{\widehat{H}}(V_{\widehat{H}})$, note that $$g \in (h_\alpha)^*_{\widehat{H}}(K'_{\widehat{H}}) = K_{\widehat{H}},$$ and so by equations~\ref{eq:h-inv} and \ref{eq:h}, $$g = (h_\alpha)^*_{\widehat{H}} \left( (h^{-1}_\alpha)^*_{\widehat{H}} (g)\right) = (r \cdot q)^*_{\widehat{H}} \varphi (g).$$  Note that since $x | (r \cdot q)$, the polynomial $r \cdot q - 1$ is nonzero, so we have a contradiction to the fact that $\varphi$ is everywhere generic.
\end{proof}

Now we return to the main proof and argue by cases.

\textbf{Case 1:} $(*)$ holds.

Fix a nonempty clopen $K \leq \G$ witnessing $(*)$.

\begin{claim}
There is an element $a \in K$ such that

(E) For any $b \in \G(M^{j(\alpha)}_\alpha)$, any $q \in D[x]$ such that $x | q$, and any $H \leq G$ which is good for $q$, if $a_H \in \dom(q^*_H \varphi),$ then $$(h_\alpha)^*_H(a_H) - q^*_H \varphi (a_H) \neq b_H;$$ and

(F) For any $c \in \cl_D (\bigcup_{\ell < \omega} (\G(M^\ell_\alpha) \cup X^\ell_\alpha))$, any nonzero $q \in D[x]$, and any $H \leq G$ good for $q$, if $a_H \in \dom(q^*_H \varphi)$, then $$q^*_H \varphi (a_H) \neq c_H.$$

\end{claim}

\begin{proof}

For any $q \in D[x]$ and any $H \leq G$ which is good for $q$, we define two subgroups $K_{q,H}$ and $K'_{q,H}$ of $K$ as follows.  If $x | q$, then let $$K_{q,H} = \pi_H^{-1} \left[\ker((h_\alpha)^*_H - q^*_H \varphi)\right] \cap K,$$ where $\pi_H: \G \rightarrow \G_H$ is the natural projection map.  Note that $(h_\alpha)^*_H - q^*_H$ is a continuous group map, so its kernel is a closed subgroup of $\G_H$, and therefore $K_{q,H}$ is a closed subgroup of $\G$.  If $x$ does not divide $q$, we let $K_{q,H} = \left\{0\right\}$.  By assumption~$(*)$, $K_{q,H}$ is always nowhere dense.  If $q \neq 0$, then let $$K'_{q,H} = \pi_H^{-1} \left[\ker(q^*_H \varphi)\right] \cap K,$$ and if $q = 0$, then let $K'_{q,H} = \left\{0\right\}$.  As before, $K'_{q,H}$ is a closed subgroup of $\G$, and the genericity of $\overline{\varphi}$ implies that $K'_{q,H}$ is nowhere dense.

So by Lemma~\ref{generic_elements} applied to the nondegenerate map $(x,y) \mapsto x - y$, there is a collection $\langle a_\sigma : \sigma < 2^{\aleph_0} \rangle$ of elements of $K$ such that for any two distinct $\sigma, \tau < 2^{\aleph_0}$, $a_\sigma - a_\tau$ does not lie in any of the groups $K_{q,H}$ or $K'_{q,H}$.  Now since $$|\cl_D \left(\bigcup_{\ell < \omega} (\G(M^\ell_\alpha) \cup X^\ell_\alpha)\right) | < 2^{\aleph_0},$$ for any $q \in D[x]$ there are fewer than $2^{\aleph_0}$ choices of $\sigma < 2^{\aleph_0}$ such that $a_\sigma$ belongs to the $K_{q,h}$-coset or the $K'_{q,H}$-coset of some element of $\cl_D (\bigcup_{\ell < \omega} (\G(M^\ell_\alpha) \cup X^\ell_\alpha))$.  Since the cofinality of $2^{\aleph_0}$ is uncountable, there is some $\sigma < 2^{\aleph_0}$ such that for \emph{any} $q \in D[x]$ and any $h \in \cl_D (\bigcup_{\ell < \omega} (\G(M^\ell_\alpha) \cup X^\ell_\alpha))$, $a_\sigma - h \notin (K_{q,H} \cup K'_{q,H})$.  Let $a = a_\sigma$; this works.

\end{proof}

Pick $a$ as in the Claim above and pick $g \in G(\mathfrak{C})$ such that $g + G^\circ = a$.  If $\ell \leq i(\alpha)$, then we let $$M^\ell_{\alpha + 1} = \acl( M^\ell_\alpha \cup \{ f^n(g) : n < \omega \}),$$ and if $\ell > i(\alpha)$, we let $$M^\ell_{\alpha+1} = \acl(M^\ell_\alpha \cup \{f^n(g) : (\ell - i(\alpha)) \leq n < \omega \}).$$  If $\ell \neq j(\alpha)$, the we let $X^\ell_{\alpha+1} = X^\ell_\alpha$, and let $X^{j(\alpha)}_{\alpha+1} = X^{j(\alpha)}_\alpha \cup \{h_\alpha (a)\}$.

We check this works.  Note that condition~5 (C) holds by definition, and conditions~1 through 3 are automatic.  For Condition~4, we first check that $\G(M^\ell_{\alpha+1}) \cap X^\ell_\alpha = \emptyset$.  First note that by Lemma~\ref{loc_mod_formula}, $$\G(M^\ell_{\alpha+1}) = \cl_D \left(\G(M^\ell_{\alpha}) \cup \left\{\varphi^n(a) : n < \omega \right\} \right).$$  So if Condition~4 fails, then there is some $b \in \G(M^\ell_\alpha)$, some $q \in D[x]$, and some $H \leq G$ good for $q$ such that $$b_H + q^*_H\varphi (g_H) \in (X^\ell_\alpha)_H,$$ but since $a_H = g_H$, this contradicts condition (F) above.  The only other way that 4 could fail is if $h_\alpha(a) \in \G(M^{j(\alpha)}_{\alpha+1})$, or equivalently (by Lemma~\ref{loc_mod_formula}), there is some $q \in D[x]$ such that $x | q$, some $H \leq G$ good for $q$, and some $b \in \G(M^{j(\alpha)}_\alpha)$ such that $$(h_\alpha (a))^*_H = (h_\alpha)^*_H (a_H) = b_H + q^*_H \varphi (a_H),$$ but this contradicts (E).

\textbf{Case 2:} $(\dagger)$ holds.

Exactly like in Case~1, we have:

\begin{claim}
There is an element $b \in \G$ such that

(E') For any $a \in \G(M^{i(\alpha)}_\alpha)$, any $q \in D[x]$, and any $H \leq G$ which is good for $q$, if $b_H \in \dom(q^*_H \varphi)$, then $$(h^{-1}_\alpha)^*_H (b) - q^*_H \varphi (b_H) \neq a_H;$$ and

(F') For any $c \in \cl_D (\bigcup_{\ell < \omega} (\G(M^\ell_\alpha) \cup X^\ell_\alpha))$, any nonzero $q \in D[x]$, and any $H \leq G$ good for $q$, if $b_H \in \dom(q^*_H \varphi)$, then $$q(\varphi)(b_H) \neq c_H.$$

\end{claim}

Then pick $g \in G(\mathfrak{C})$ such that $g + G^\circ = b$, and let $$M^\ell_{\alpha+1} = \acl (M^\ell_\alpha \cup \left\{f^n(g) : n < \omega \right\}),$$ $$N_{\alpha + 1} = \acl (N_\alpha \cup \left\{f^n(h) : n < \omega \right\}),$$ let $X^\ell_{\alpha+1} = X^\ell_\alpha$ if $\ell \neq i(\alpha)$, and let $X^{i(\alpha)}_{\alpha+1} = X^{i(\alpha)}_\alpha \cup \left\{h^{-1}_\alpha(b)\right\}$.

Just as before, it can be checked that these sets satisfy conditions~1 through 5.

\end{proof}

\begin{remark}
Since for the base case of our construction, we let the models $M^\ell_0$ (for $\ell < \omega$) all be equal, it is worth pointing out why the models $M^\ell$ we eventually construct are not all equal.  To see this, note that the identity map from $\overline{G}$ to $\overline{G}$ will be listed (infinitely often) as some $h_\alpha$, and so the fact that $f$ has independent orbits implies that $(*)$ holds (with $K = \overline{G}$).  Thus at stage $\alpha$, we ensure that $M^{i(\alpha)} \neq M^{j(\alpha)}$.
\end{remark}

\section{Weakly minimal theories}

In this section, we return to the general context of weakly minimal theories and prove Theorem~\ref{wm_SB}.  We assume throughout this section that $T$ is a weakly minimal theory, and ``$a \in M$'' means that $a$ is in the home sort of $M$.  These assumptions imply that for any $a \in M$, the type $\tp(a/\emptyset)$ is \emph{minimal} (that is, has $U$-rank $1$), though it is not necessarily weakly~minimal.

\begin{definition}
Let $p \in S(A)$ be a minimal type.
\begin{enumerate}
\item The type $p$ \emph{has bounded orbits over $B$} if there is some $n < \omega$ such that for every $f \in \Aut(\mathfrak{C} / B)$, there are $i < j \leq n$ such that $f^i(p) = f^j(p)$.  Otherwise, $p$ \emph{has unbounded orbits} over $B$.
\item The type $p$ \emph{has dependent orbits over $B$} if for every $f \in \Aut(\mathfrak{C} / B)$, there is an $n < \omega$ such that $$f^n(p) \, {\not\perp}^a \, p \otimes f(p) \otimes \ldots \otimes f^{n-1}(p).$$  Otherwise, $p$ \emph{has an independent orbit over $B$.}
\item If $p \in S(\acl(\emptyset))$, then we say $p$ has bounded orbits (or dependent orbits) if it has bounded (dependent) orbits over $\emptyset$.
\end{enumerate}
\end{definition}

\begin{remark}
The minimal type $p \in S(\acl(\emptyset))$ has an independent orbit if and only if there is an $f \in \Aut(\mathfrak{C})$ such that for \emph{any} choice of realizations $\langle a_i : i < \omega \rangle$ of the types $f^i(p)$, the set $\{ a_i : i < \omega \}$ is independent.
\end{remark}

\begin{quest}
1. If $p$ has unbounded orbits, then does $p$ necessarily have an independent orbit?

2. If $p \in S(\acl(\emptyset))$ has dependent orbits and $g \in \Aut(\mathfrak{C})$, then does $g(p)$ also have dependent orbits?
\end{quest}

Note that in the terminology above, if $p$ is the generic type of a weakly~minimal, locally modular group $G$ defined over $\acl(\emptyset)$, then $\Aut(G)$ does not have a weakly~generic element if and only if for \emph{every} generic type $q$ of $G$ has dependent orbits.  

If $p \in S(\acl(\emptyset))$ is minimal and $M \models T$, let ``$\dim(p,M)$'' mean the dimension of $p(M)$ as a pregeometry.

\begin{definition}
An elementary map $f: M \rightarrow N$ between models of $T$ is called \emph{dimension-preserving} if for any minimal $p \in S(\acl(\emptyset))$, $\dim(p,M) = \dim(f(p),N)$.
\end{definition}

\begin{thm}
\label{dimension_preserving}
Any two models $M, N$ of $T$ are isomorphic if and only if there is a dimension-preserving map $f: M \rightarrow N$.
\end{thm}

\begin{proof}
Left to right is obvious, since any isomorphism is dimension-preserving.  
For the converse, suppose that $f: M \rightarrow N$ is dimension-preserving. 

If $A \subseteq M$, we say that $A$ is \emph{type-closed (in $M$)} if $\acl(\emptyset)\subseteq A$
and whenever $a \in A$, $a' \in M$, and $\stp(a') = \stp(a)$ is minimal, then $a' \in A$.

If $A \subseteq M$ and $B \subseteq N$, we call an elementary map $g: A \rightarrow B$ \emph{closed} if its domain is type-closed in $M$ and its image is type-closed in $N$.

If $A \subseteq \mathfrak{C}$ and $p \in S(\acl(\emptyset))$, then define 
$k(p, A) \in \omega$ to be the largest $k$ (if one exists) such that $p^{(k)}$ is almost orthogonal to $\tp(A/\acl(\emptyset))$.
Note that  $A$ is independent from any Morley sequence in $p$ of length at most $k(p,A)$ (if it exists), and 
$A$ is independent from every Morley sequence in $p$ if $k(p,A)$ does not exist.

\begin{claim}
\label{basis}
Suppose that $A \subseteq M$ is type-closed, $p \in S(\acl(\emptyset))$ is minimal, and $k(p, A)$
exists.  Then for any Morley sequence $I_0 \subseteq M$ in $p$ of length $k(p,A)$, $p(M) \subseteq \acl(A \cup I_0)$.
\end{claim}

\begin{proof}
By definition of $k(p, A)$, there is some $a \models p | I_0$ such that $a \nind_{I_0} A$.  Pick such an $a$ and a finite set $A_0 \cup \{c\} \subseteq A$ such that $I_0 a \nind A_0 c$.  We may assume that the size of $A_0 \cup \{c\}$ is minimal, so that $A_0$ is independent from any Morley sequence in $p$ of length $k(p,A) + 1$.  By the weak minimality of $T$ and stationarity, if $q = \stp(c)$, then \emph{any} realization of $p | I_0$ is interalgebraic over $I_0 \cup A_0$ with some realization of $q | A_0$.  So any $b \in p(M) \setminus \acl(I_0)$ is interalgebraic over $A_0 \cup I_0$ with some $d \in q(M)$, and the fact that $A$ is type-closed implies that $d \in A$.
\end{proof}

It follows from Claim~\ref{basis} (when $k(p,A)=0$) that if $A\subseteq M$ is type-closed, then so is $\acl(A)$.
Thus, any elementary $h:\acl(A)\rightarrow\acl(B)$ extending a closed map $g:A\rightarrow B$ is closed.

Visibly, the restriction $\sigma:=f\upharpoonright\acl(\emptyset)$ is closed, and the union of an increasing sequence of
closed maps is closed.  Thus, to conclude that $M$ and $N$ are isomorphic, it suffices by Zorn's Lemma and symmetry
to show
that if $g:A\rightarrow B$ extending $\sigma$ is closed, and $p\in S(\acl(\emptyset))$ is minimal, 
then there is a closed $h$ extending $g$, whose domain contains $A\cup p(M)$.

Fix such a $g$ and $p$.  Note that $k(p,A)$ exists if and only if $k(g(p),B)$ exists and when they exist, they are equal.
There are three cases:

Case~1:  $k(p,A)$ does not exist.  Let $I\subseteq M$ be any maximal Morley sequence in $p$, and $J\subseteq N$ be any
maximal Morley sequence in $g(p)$.  Since $g$ extends $\sigma$, $g$ is dimension preserving, so $|I|=|J|$.
Let $j:I\rightarrow J$ be any bijection.  Then $g\cup j$ is elementary, and any elementary map  
$h:\acl(A\cup I)\rightarrow\acl(B\cup J)$ 
extending $g\cup j$ will be closed.

Case~2: $k(p,A)$ exists, but $\dim(p,M)\le k(p,A)$.  Then again, $\dim(p,M)=\dim(g(p),N)$, and for any bijection $j$ between
bases for $p(M)$ and $g(p)(M)$, there is a closed $h:\acl(AI)\rightarrow\acl(BJ)$ extending $g\cup j$.

Case~3:  $k(p,A)$ exists, and $\dim(p,M)> k(p,A)$.  Let $I_0\subseteq p(M)$ and $J_0\subseteq g(p)(N)$
be  Morley sequences of length $k(p,A)$ in $p$ and $g(p)$, respectively, and let $j:I_0\rightarrow J_0$ be
any bijection.  Again, $g\cup j$ is elementary and any elementary $h:\acl(A\cup I_0)\rightarrow \acl(B\cup J_0)$
extending $g\cup j$ is closed.  But   $p(M)\subseteq\acl(A\cup I_0)$ and
$g(p)(M)\subseteq\acl(B\cup J_0)$ by Claim~\ref{basis}, so any such $h$ suffices.

\end{proof}

\begin{cor}
\label{SB_dep_orbits}
If $T$ is weakly minimal and every minimal type $p \in S(\acl(\emptyset))$ has dependent orbits, then $T$ has the SB~property.
\end{cor}

\begin{proof}
If $M$ and $N$ are models of such a theory and $f: M \rightarrow N$ and $g : N \rightarrow M$ are elementary embeddings, then $f$ must be dimension-preserving.  So by Theorem~\ref{dimension_preserving}, $M \cong N$.
\end{proof}

The next result is from the first author's thesis (\cite{my_thesis}), where types as in the hypothesis were called ``nomadic.''

\begin{thm}
\label{nomadic_types}
If $T$ is stable and there is a stationary regular type $p \in S(A)$ and $f \in \Aut(\mathfrak{C})$ such that the types $\{f^i(p) : i < \omega\}$ are pairwise orthogonal, then $T$ has an infinite collection of models that are pairwise nonisomorphic and pairwise not bi-embeddable.
\end{thm}

\begin{proof}
Just for simplicity of notation, we will assume that $T$ is countable and superstable, but the same argument works in general if we lengthen our Morley sequences a bit and replace the $\textbf{F}^a_{\aleph_0}$-prime models by $\textbf{F}^a_\kappa$-prime models for some suitably large $\kappa$.  We will use without proof some well-known facts about $a$-prime models which are proved in section~1.4 of \cite{pillay} and in chapter~IV of \cite{bible} (in the latter reference, they are called ``$\textbf{F}^a_{\aleph_0}$-prime models'').

Pick $p$ and $f$ as in the hypothesis.  Since $T$ is superstable, we may assume that $A$ is countable.  Let $A_i = f^i(p)$ and $p_i = f^i(p)$ (which is in $S(A_i)$), and let $B = \bigcup_{i < \omega} A_i$.  Pick sequences $\langle I^j_i : i, j < \omega \rangle$ such that $I^j_i$ is a Morley sequence in $p_i | B$ of length $\aleph_{i + j + 1}$ and $$I^j_i \ind_B \bigcup_{k \neq i} I^j_k .$$  For each $j < \omega,$ let $M_j$ be an $a$-prime model over $B \cup \bigcup_{i < \omega} I^j_i$.

By using $a$-primeness and iterating the map $f$, it is straightforward to check that the models $\langle M_j : j < \omega \rangle$ are pairwise bi-embeddable.  To prove they are nonisomorphic, we set some notation.  If $p \in S(C)$ is a regular type and $C \subseteq M \models T$, then $\dim(p, M)$ is the cardinality of a maximal Morley sequence in $p$ inside $M$; as noted in section~1.4.5 of \cite{pillay}, this is well-defined for any model $M$.  

\begin{claim}
\label{dimension_bound}
If $q \in S(C)$ is any regular type and $M$ is an $a$-prime model over $C$, then $\dim(q,M) \leq \aleph_0$.
\end{claim}

\begin{proof}
This is a special case of Theorem~IV.4.9(5) from \cite{bible}.
\end{proof}

\begin{claim}
If $q \in S(C)$ is any regular stationary type over a countable set $C \subseteq M_j$, then either $\dim(q, M_j) \leq \aleph_0$ or $\dim(q, M_j) \geq \aleph_{j+1}$.
\end{claim}

\begin{proof}

Case 1: For some $i < \omega$, $q \not\perp p_i$.

Pick some $N$ which is $a$-prime over $B \cup C$, and by primeness we may assume $N \prec M_j$.  Since $N$ is an $a$-model, $q | N$ is domination equivalent to $p_i | N$, and so $\dim(q | N, M_j) = \dim(p_i | N, M_j)$.  If $J \subseteq N$ is a maximal Morley sequence in $p_i$, then each $c \in J$ forks with $C$ over $B$, so $|J| \leq \aleph_0$; therefore $\dim(p_i, N) \leq \aleph_0$.  By Lemma~1.4.5.10 of \cite{pillay}, $$\dim(p_i, M_j) = \dim(p_i, N) + \dim(p_i | N, M_j);$$ so since $\dim(p_i, M_j) \geq \aleph_{i+j+1}$, we must have $\dim(p_i | N, M_j) \geq \aleph_{i+j+1}$.  So $\dim(q, M_j) \geq \dim(q | N, M_j) = \dim(p_i | N, M_j) \geq \aleph_{j+1}$. 

Case 2: For every $i < \omega$, $q \perp p_i$.

Pick some finite $D \subseteq C$ such that $q$ does not fork over $D$.  If $J \subseteq M_j$ is a maximal Morley sequence in $q$, then by standard forking calculus we can find a subset $\widetilde{J}$ of $J$ such that $\widetilde{J} \ind_D B$, $\widetilde{J} \ind_{BD} I^j_{<\omega},$ and $|J \setminus \widetilde{J} | \leq \aleph_0$.  Since $D$ is finite, the model $M_j$ is $a$-prime over $B \cup D \cup I^j_{< \omega}$, and so by Claim~\ref{dimension_bound}, $|\widetilde{J}| \leq \aleph_0$.  Therefore $|J| \leq \aleph_0$ and we are done.


\end{proof}

If $j < k < \omega$, then the nonisomorphism of $M_j$ and $M_k$ follows from the previous claim plus:

\begin{claim}
$\dim(p, M_j) = \aleph_{j+1}$.
\end{claim}

\begin{proof}
Suppose $J \subseteq M_j$ is a maximal Morley sequence in $p$.  Without loss of generality, $J \supseteq I^j_0$, so $|J| \geq \aleph_{j+1}$.  Since $B$ is countable, there is a countable $J_0 \subseteq J$ such that $J_1 := J \setminus J_0$ is independent over $B$.  Since $p$ is orthogonal to every type $p_i$ for $i > 0$, $J_1$ is independent over $B \cup \bigcup_{0 < i < \omega} I^j_i$.  So if $J_2 = J_1 \setminus I^j_0$, then $J_2$ is Morley over $B \cup \bigcup_{i < \omega} I^j_i$.  But since $M_j$ is $a$-constructible over $B \cup \bigcup_{i < \omega} I^j_i$, it follows (by Claim~\ref{dimension_bound}) that $|J_2| \leq \aleph_0$, and thus $|J| = \aleph_{j+1}$.
\end{proof}

\end{proof}

\textit{Proof of Theorem~\ref{wm_SB}:}
2 $\Rightarrow$ 1 was Corollary~\ref{SB_dep_orbits}, and 1 $\Rightarrow$ 3 is immediate.  All that is left is to show that if 2 fails then so does 3.

So suppose $T$ has a minimal type $p \in S(\acl(\emptyset))$ with an independent orbit, and say $$p \, \bot^a \, \, f(p) \otimes f^2(p) \otimes \ldots$$ where $f \in \Aut(\mathfrak{C})$.  Then $p$ cannot be strongly minimal, so by Buechler's dichotomy, $p$ must be locally~modular.  

First suppose that $p$ is nontrivial.  If $c$ is any realization of $p$, then each of the types in $\{f^i(p) : 1 \leq i < \omega \}$ is non-almost-orthogonal over $c$ to a generic type $q \in S(\acl(\emptyset))$ of some weakly~minimal, $\acl(\emptyset)$-definable group $G$ (see \cite{almost_orthogonality} or \cite{pillay}).  By Fact~\ref{abelian_by_finite}, we may assume that $G$ is abelian.  We temporarily add constants to the language for the algebraic parameters used to define $G$ so that $G$ is definable over $\emptyset$, and let $T'$ be this expanded language.  There must be some finite power $f^k$ of $f$ which fixes these parameters, so without loss of generality $f$ is still an automorphism in the language of $T'$.  It follows that $q$ also has an independent orbit, as witnessed by $f$ again, and so $\G$ has an everywhere generic automorphism.  By Theorem~\ref{wm_group_SB}, the theory $T'$ has infinitely many pairwise bi-embeddable, pairwise nonisomorphic models, and by Lemma~\ref{alg_parameters}, so does the original theory $T$.

Finally, suppose that $p$ is trivial.  Then the types in $\{f^i(p) : i < \omega\}$ are pairwise orthogonal (see \cite{trivial_considerations}).  By Theorem~\ref{nomadic_types}, we are done. $\square$

\begin{remark}
It seems that the Dushnik-Miller argument used in section~3 for weakly minimal groups could also be applied to weakly minimal theories in which there is a trivial type with an independent orbit, yielding a more uniform proof of Theorem~\ref{wm_SB} which avoids the construction in the proof of Theorem~\ref{nomadic_types}.  There are some technical obstacles to doing this, however, so we have not included this argument in the present write-up.
\end{remark}


\begin{thebibliography}{9}


\bibitem{trivial_considerations} John B. Goode, ``Some trivial considerations," \emph{Journal of Symbolic Logic} \textbf{56} (1991), number~2, 624--631.

\bibitem{dushnik_miller} Ben Dushnik and E.\ W.\ Miller, ``Partially ordered sets," \emph{American Journal of Mathematics} \textbf{63} (1941), number~3, 600--610.

\bibitem{my_thesis} John Goodrick, ``When are elementarily bi-embeddable structures isomorphic?," Ph.\ D.\ thesis, University of California, Berkeley, 2007.

\bibitem{almost_orthogonality} Ehud Hrushovski, ``Almost orthogonal regular types," \emph{Annals of Pure and Applied Logic} \textbf{45} (1939), 139--155.

\bibitem{jech} Thomas Jech, \emph{Set Theory}, Academic Press, 1978.

\bibitem{nur1} T.\ A.\ Nurmagambetov, ``Characterization of $\omega$-stable theories with a bounded number of dimensions," \emph{Algebra i Logika} \textbf{28} (1989), number~5, 388--396.

\bibitem{nur2} T.\ A.\ Nurmagambetov, ``The mutual embeddability of models," from \emph{Theory of Algebraic Structures} (in Russian), Karagand.\ Gos.\ Univ.\, 1985, pp.\ 109--115.

\bibitem{pillay} Anand Pillay, \emph{Geometric Stability Theory}, Oxford University Press, 1996.

\bibitem{bible} Saharon Shelah, \emph{Classification Theory} (2nd edition), North-Holland (1990).


\end{thebibliography}
\end{document}